\documentclass[12pt]{amsart}

 \usepackage{amsfonts,graphics,amsmath,amsthm,amsfonts,amscd,amssymb,amsmath,latexsym,multicol,
 mathrsfs}
\usepackage{epsfig,url}
\usepackage{flafter}
\usepackage{fancyhdr}
\usepackage{color}
\usepackage{hyperref}
\hypersetup{colorlinks=true, linkcolor=black}
\usepackage{ulem}

\makeatletter

\def\jobis#1{FF\fi
  \def\predicate{#1}%
  \edef\predicate{\expandafter\strip@prefix\meaning\predicate}%
  \edef\job{\jobname}%
  \ifx\job\predicate
}

\makeatother

\if\jobis{proposal}%
\else
\fi

 \usepackage[matrix, arrow]{xy}

\DeclareMathOperator{\Supp}{Supp}

\DeclareMathOperator{\Spec}{Spec}

\DeclareMathOperator{\Pic}{Pic}

 \numberwithin{equation}{subsection}
 \numberwithin{footnote}{subsection}

 \newtheorem{cor}[subsection]{Corollary}
 \newtheorem{lem}[subsection]{Lemma}
 \newtheorem{prop}[subsection]{Proposition}
 \newtheorem{thm}[subsection]{Theorem}
 \newtheorem{conj}[subsection]{Conjecture}

{
    \newtheoremstyle{upright}%
        {8pt plus2pt minus4pt}%
        {8pt plus2pt minus4pt}%
        {\upshape}%
        {}%
        {\bfseries\scshape}%
        {}%
        {1em}%
        {}%
\theoremstyle{upright}

 \newtheorem{rem}[subsection]{Remark}

}

 \newcommand{\PP}{\mathbb P}
 
 \newcommand{\Q}{\mathbb Q}
 \newcommand{\R}{\mathbb R}
 \newcommand{\Z}{\mathbb Z}
 \newcommand{\bir}{\dashrightarrow}
 \newcommand{\rddown}[1]{\left\lfloor{#1}\right\rfloor} 

\title{Iitaka's $C_{n,m}$ conjecture for $3$-folds over finite fields}
\author{Caucher Birkar, Yifei Chen, and Lei Zhang}
\date{\today}
\begin{document}
\maketitle

\begin{abstract}
We prove Iitaka's $C_{n,m}$ conjecture for $3$-folds over the algebraic closure of finite fields.
Along the way we prove some results on the birational geometry of log surfaces over nonclosed fields
and apply these to existence of relative good minimal models of $3$-folds.
\end{abstract}

\tableofcontents


\section{Introduction}

\subsection*{Iitaka's conjecture}
Let $X$ be a normal projective variety over a field $k$, $L$  a Cartier divisor on $X$, and $N(L)$ the set of all positive integers $m$ such that the linear system $|mL| \neq \emptyset$. For an integer $m \in N(L)$, let $\Phi_{|mL|}$ be the rational map defined by  $|mL|$. The Kodaira dimension
$\kappa(L)$ is defined as
\[\kappa(L) =\left\{
\begin{array}{lll}
-\infty,&\text{if }& N(L) = \emptyset\\
\max \{\dim \Phi_{|mL|}(X)\mid m \in N(L) \} &\text{if}& N(L) \neq \emptyset
\end{array}\right.
\]
If $L$ is a $\Q$-Cartier divisor, $\kappa(L):=\kappa(mL)$ for any natural number $m$ so that $mL$ is Cartier.
This does not depend on the choice of $m$.

The following conjecture due to Iitaka (in characteristic zero) is of fundamental importance in the classification theory of algebraic varieties.

\begin{conj}[$C_{n,m}$]\label{conj-iitaka}
Let $f\colon X\to Z$ be a contraction between smooth projective varieties
of dimension $n,m$ respectively, over an
algebraically closed field $k$. Assume the generic fibre $F$ is smooth.   Then
$$
\kappa(K_X)\ge \kappa(K_F)+\kappa(K_Z)
$$
\end{conj}

One can formulate a more general problem when $F$ is not smooth either by assuming it is geometrically integral 
with a resolution or by considering $F$ as a variety over the function field of $Z$ (for example, see Corollary \ref{c-main}).

Over the complex numbers, the conjecture has been studied by Kawamata [\ref{Ka0}][\ref{Ka1}][\ref{Ka2}],  Koll\'ar [\ref{Ko}], Viehweg [\ref{Vie1}][\ref{Vie2}][\ref{Vie2}], Birkar [\ref{bi09}], Chen and Hacon [\ref{CH}], Cao and P\v aun [\ref{Cao-Paun}], etc. We refer to [\ref{cz13}] for a collection of results over $\mathbb{C}$.  In positive characteristic, Chen and Zhang proved the conjecture for fibrations of relative dimension one [\ref{cz13}], and Patakfalvi proved it when $Z$ is of general type and the generic geometric fibre satisfies certain properties [\ref{pa13}, Theorem 1.1] (see also [\ref{pa12}, Corollary 4.6]).

In this paper, we will prove:

\begin{thm}\label{t-main}
Conjecture $C_{n,m}$ holds when $n=3$, $k = \bar{\mathbb{F}}_p$, and $p > 5$.
\end{thm}

The case $C_{3,2}$ follows from [\ref{cz13}], so the main result here is $C_{3,1}$.
Our main tools are the log minimal model program for $3$-folds developed recently by Hacon, Xu, and Birkar [\ref{hx13}][\ref{bi13}][\ref{xu13}], birational geometry of log surfaces over nonclosed fields (see below), and the semi-positivity results of Patakfalvi [\ref{pa12}]. The reason for the restriction $k = \bar{\mathbb{F}}_p$ is that it is often easier to prove
semi-ampleness of divisors over finite fields; for example if $K_X\sim_\Q f^* D$ for some $D\equiv 0$ on $Z$,
then $D\sim_\Q 0$ is automatic over $k = \bar{\mathbb{F}}_p$ but the same conclusion would perhaps require substantial effort
over other fields; this is a major issue also in characteristic zero [\ref{Ka1}].\\

Since resolution theory holds in dimension three in positive characteristics, we get
\begin{cor}\label{c-main}
Let $f\colon X\to Z$ be a contraction, from a smooth projective three dimensional variety to a smooth projective curve over $\bar{\mathbb{F}}_p, p>5$. Let $\tilde{F}$ be a smooth model of generic geometric fibre of $f$. Then
$$
\kappa(K_X)\geq \kappa(K_{\tilde{F}})+\kappa(K_Z)
$$
\end{cor}

\subsection*{Log surfaces over nonclosed fields}

Let $X\to Z$ be a contraction between normal varieties and let $F$ be its generic fibre.
As is well-known, in char $p>0$, $F$ may not be smooth even if $X$ and $Z$ are smooth.
Actually $F$ may even be geometrically non-reduced. This creates difficulties because proofs in
birational geometry are often based on induction and in this case we cannot simply apply induction
and lift information from $F$ to $X$. On the other hand, $F$ has nice properties if
we think of it as a variety over the function field of $Z$ without passing to the algebraic closure of this function field.
For example, if $X$ is smooth, then $F$ is regular.
In particular, relevant to this paper is the case in which $X$ is a $3$-fold and $Z$ is a curve.
So it is natural for us to consider surfaces over a not necessarily algebraically closed field $k$.

It is easy to define pairs, singularities, minimal models, etc over an arbitrary field.
See \ref{d-pairs} and \ref{d-mmodel} for more details.

\begin{thm}\label{t-lmmp-nonclosed-surf}
Let $(X,B)$ be a projective dlt pair of dimension two over a field $k$ where $B$ is a $\Q$-boundary.
Then we can run an LMMP on $K_X+B$ which ends with a log minimal model or a Mori fibre space.
\end{thm}

\begin{thm}\label{t-abund-nonclosed-surf}
Let $(X,B)$ be a projective klt pair of dimension two over a field $k$  where $B$ is a $\Q$-boundary.
Assume $K_X+B$ is nef and that $\kappa(K_X+B)\ge 0$. Then $K_X+B$ is semi-ample.
\end{thm}

These results were proved by Tanaka [\ref{ta15-2}] not long ago. Actually he proves more general statements, 
in particular, he proves \ref{t-abund-nonclosed-surf} without the assumption $\kappa(K_X+B)\ge 0$. 
We give a self-contained proof of the above theorems. Our proof of \ref{t-lmmp-nonclosed-surf} is perhaps  
 the same as that in [\ref{ta15-2}] which closely follows Keel's techniques [\ref{Ke99}]. 
However, our proof of \ref{t-abund-nonclosed-surf} seems to be different
from his. He relies on another paper [\ref{ta15-1}] but our proof is short and direct
which follows Mumford's ideas [\ref{Mum}] and uses a result of Totaro [\ref{tot09}]. In fact we worked out 
these proofs before [\ref{ta15-2}] appeared.\\

\subsection*{Relative good minimal models of $3$-folds}
As mentioned earlier our motivation for considering surfaces over nonclosed fields is
to treat $3$-folds over curves.

\begin{thm}\label{t-3d-rel-mmodel}
Let $(X,B)$ be a projective klt pair of dimension $3$  where $B$ is a $\Q$-boundary,
and $f\colon X\to Z$ be a contraction onto a curve,
over $\bar{\mathbb{F}}_p$ with $p>5$. Let $F$ be the generic fibre of $f$.
If $\kappa((K_X+B)|_F)\ge 0$, then $(X,B)$ has a good log minimal model over $Z$.
\end{thm}

Actually the proof of the theorem works over any algebraically closed field of char $p>5$
except when  $\kappa((K_X+B)|_F)=1$. In this case we make use of the fact that any nef and big divisor
on a surface over $\bar{\mathbb{F}}_p$ is semi-ample.

As far as Theorem \ref{t-main} is concerned we only need special cases of \ref{t-3d-rel-mmodel} which in turn
only needs special cases of \ref{t-abund-nonclosed-surf}. We only need the case when $B=0$ and $F$ is smooth, or
when $\kappa((K_X+B)|_F)= 0$ but $F$ admits a contraction onto an elliptic curve. See Remark \ref{r-nonclosed-surf}
for some more detailed explanations.\\

\subsection*{Acknowledgements}
The first author was supported by a grant of the Leverhulme Trust.
Part of this work was done when the first author visited Mihai P\u{a}un in KIAS in February 2015.
Part of this work was done when the first author visited National Taiwan University in August-September 2014 with the support of the 
Mathematics Division (Taipei Office) of the National Center for Theoretical Sciences, and the visit was arranged by Jungkai A. Chen. 
He wishes to thank them all for their hospitality. The second author was supported by NSFC (No. 11201454 and No. 11231003). The third author was supported by grant
NSFC (No. 11401358). We would like to thank Burt Totaro for
answering our questions regarding the results in his paper [\ref{tot09}].

\section{Preliminaries}

We follow Koll\'ar [\ref{Kollar-singularities}] to define canonical sheaves and divisors, adjunction,
pairs, singularities, etc which we discuss below.

\subsection{Relative canonical sheaves}
Let $f\colon X\to Z$ be a morphism of schemes where $Z$ is
 regular and excellent, and $X$ is pure dimensional and of finite type over $Z$.
Let $S$ be a closed subscheme of $X$ and $U:=X\setminus S$. Assume that

$\bullet$ codimension of $S$ in $X$ is at least $2$, and 

$\bullet$ $U$ is a locally closed local complete intersection in some $\PP^n_Z$.

Let $\mathcal{I}$ be the ideal sheaf of the closure of $U$ in $\PP^n_Z$ and
let $j\colon U \to X$ be the inclusion map of $U$ in $X$. Now define the relative {canonical sheaf}
as
$$
\omega_{X/Z}=j_*((\omega_{\PP^n_Z/Z}\otimes (\det \mathcal{I}/\mathcal{I}^2)^\vee)|_U)
$$
Note that $\mathcal{I}/\mathcal{I}^2$ is locally free on $U$. Moreover,
$\omega_{\PP^n_Z/Z}$ is as usual defined to be $\mathcal{O}_{\PP^n_Z}(-n-1)$.

\subsection{Relative canonical sheaves and divisors of normal schemes}\label{ss-can-sheaf-div-normal}
Let $f\colon X\to Z$ be a quasi-projective morphism of schemes where $Z$ is
 regular and excellent, and $X$ is integral and normal.
The set of regular points of $X$ is an open subset of $X$ by definition of excellent schemes
(cf. [\ref{bh05}, page 382]).
Let $S$ be any closed subscheme of $X$ containing the singular points and such that
the codimension of $S$ in $X$ is at least $2$. Such an $S$ exists because $X$ is normal.
If we embed $U$ as a locally closed subscheme
into some $\PP^n_Z$, then $U$ is a locally closed local complete intersection because
$U$ is regular (cf. [\ref{bh05}, Proposition 2.2.4]). Therefore we can
define the relative canonical sheaf $\omega_{X/Z}$ as in the previous subsection.
Under our assumptions, this sheaf is of the form $\mathcal{O}_X(K_{X/Z})$
for some divisor $K_{X/Z}$ which we refer to as the {canonical divisor}
of $X$ over $Z$ (when $Z$ is the spectrum of a field, we usually drop $Z$ and just write
$\omega_X$ and $K_X$ if the ground field is obvious from the context).

If $Y\to Z$ is another quasi-projective morphism from a  normal integral scheme $Y$ with 
$K_{Y/Z}$ being $\Q$-Cartier, and if we are given a morphism $h\colon X\to Y/Z$, then we let 
$K_{X/Y}=K_{X/Z}-h^*K_{Y/Z}$. 

Now assume that $Z$ is integral and let $F$ be the generic fibre of $X\to Z$.
Let $V$ and $T$ be the inverse images of $U$ and $S$ under the morphism $F\to X$.
Then by Lemma \ref{l-gen-fib} below, $F$ is normal, $V$ is regular, and the codimension of
$T$ in $F$ is at least $2$. We consider $F$ with its natural scheme structure
over $K$, the function field of $Z$.
By the definition of canonical sheaves,
$\omega_{V}$ is the pullback of $\omega_{U/Z}$. Therefore, $K_{V}$ is the pullback of $K_{U/Z}$.
Moreover, if $K_X+B$ is $\R$-Cartier for some $\R$-divisor $B$, then we can
write $K_{F}+B_F$ for the pullback of $K_X+B$ to $F$ where $B_F$ is canonically
determined by $B$: more precisely, $B_F$ is the closure of the pullback of
$B|_U$ to $V$.

\subsection{Intersection theory}\label{ss-intersection}
For a short introduction to intersection theory on a proper scheme $X$ over a field $k$, see [\ref{KM}, Section 1.5].
Note that intersection numbers depend on the ground field $k$. For a detailed treatment of
intersection theory on regular surfaces, see [\ref{Liu}, Chapter 9]. Although [\ref{Liu}] does not seem to
treat the Riemann-Roch formula but it holds on regular projective surfaces. More precisely, if $X$ is a regular
surface projective over a field $k$ and if $L$ is a Cartier divisor, then
$$
\mathcal{X}(mL)=\frac{1}{2}L\cdot (L-K_X)+\mathcal{X}(0)
$$
where $K_X$ means the relative canonical divisor of $X$ over $k$, and
$\mathcal{X}(N):=h^0(N)-h^1(N)+h^2(N)$ for any divisor (or sheaf) $N$
which also depends on the ground field $k$. The formula can be proved as in the
case of smooth surfaces over algebraically closed fields. The main point is that it can be reduced
to Riemann-Roch on curves which holds in a quite general setting (cf. [\ref{Liu}, Section 7.3]).
See [\ref{ta15-2}, Section 1.3] for a complete proof.

\subsection{Pairs and singularities}\label{d-pairs}
 Let $k$ be a field.
A pair $(X,B)$ over $k$ consists of a
normal quasi-projective variety $X$ over $k$
and an $\R$-Weil divisor $B$ with coefficients in $[0,1]$
such that $K_{X}+B$ is $\R$-Cartier. We usually refer to $B$ as a boundary and when it has rational coefficients
we say it is a $\Q$-boundary. See [\ref{Kollar-singularities}, Definitions 1.5 and 2.8] for definitions 
in more general settings.

For any projective birational morphism $f\colon W\to X$ from a normal variety $W$, we can
write $K_W+B_W=f^*(K_X+B)$ for some unique divisor $B_W$. For a prime divisor $D$ on $W$
we define the log discrepancy $a(D,X,B)$ to be $1-b$ where $b$ is the coefficient of $D$ in $B_W$.
We say $(X,B)$ is lc (resp. klt) if $a(D,X,B)\ge 0$ (resp.  $a(D,X,B)>0$) for any $D$ on any
such $W$. On the other hand, we say $(X,B)$ is dlt if there is a closed subset $Z\subset X$ of
codimension at least two such that $a(D,X,B)>0$ for any $D$ whose image in $X$ is inside $Z$
and such that outside $Z$ we have: $X$ is regular and
$\Supp B$ has simple normal crossing singularities.

We say $f$ is a log resolution if
$W$ is regular and $\Supp B_W$ has simple normal crossing singularities. Log
resolutions exist when $\dim X=2$ [\ref{shaf}] or if $k$ is algebraically closed and $\dim X\le 3$;
in these situations one can check whether $(X,B)$ is lc or klt by looking at
one log resolution. Moreover, if $\dim X=2$, then a minimal resolution of $X$ exists.

\subsection{Minimal models and Mori fibre spaces}\label{d-mmodel}
Let $(X,B)$ be a lc pair and $(Y,B_Y)$ be a $\Q$-factorial dlt pair, over a field $k$, equipped with projective morphisms
$X\to Z$ and $Y\to Z$ and a birational
map $\phi\colon X\bir Y$ commuting with these morphisms such that $\phi_*B=B_Y$ and such that $\phi^{-1}$ does not
contract divisors. Assume in addition that
$$
a(D,X,B)\le a(D,Y,B_Y)
$$
for any prime divisor $D$ on birational models of $X$ with strict inequality if $D$ is on $X$ and exceptional/$Y$.
We say $(Y,B_Y)$ is a log minimal model of $(X,B)$ over $Z$ if
$K_Y+B_Y$ is nef$/Z$. We say $(Y,B_Y)$ is a Mori fibre space of $(X,B)$
over $Z$ if there is a $K_Y+B_Y$-negative extremal contraction $Y\to T/Z$ with
$\dim Y>\dim T$.

\subsection{Minimal models of 3-folds}
For $3$-folds we have the following result.

\begin{thm}\label{t-mmp-3fold}
Let $(X,B)$ be a projective klt pair of dimension $3$ and $X\to Z$ a contraction, over an algebraically closed field
$k$ of char $p>5$.

$(1)$ If $K_X+B$ is pseudo-effective over $Z$, then $(X,B)$ has a log minimal model over $Z$;

$(2)$ If $K_X+B$ is not pseudo-effective over $Z$, then $(X,B)$ has a Mori fibre space over $Z$;

$(2)$ If $K_X+B$ is nef over $Z$, and $K_X+B$ or $B$ is big over $Z$, then $K_X+B$ is semi-ample over $Z$.
\end{thm}

Parts (1) is proved in [\ref{hx13}] for canonical singularities, and in [\ref{bi13}] in general.
Parts (2) is proved in [\ref{CTX}] for terminal singularities, and in [\ref{bw14}] in general.
Part (3) is proved in various forms in [\ref{bi13}][\ref{xu13}][\ref{bw14}].

\subsection{Adjunction}\label{ss-adjunction}
Let $X$ be a normal projective variety over a field $k$. Let $B\ge 0$ be a
$\Q$-divisor on $X$ such that $K_X+B$ is $\Q$-Cartier.
Let $S$ be a component of $\rddown{B}$. Then we can write the pullback of $K_X+B$ to the normalization
${S^\nu}$ as $K_{S^\nu}+B_{S^\nu}$ where the {different} $B_{S^\nu}\ge 0$ is
canonically determined. If $(X,B)$ is lc outside a codimension $\ge 3$ subset of
$X$, then $B_{S^\nu}$ is a boundary. See [\ref{Kollar-singularities}, Proposition 4.5]
for more details.

\subsection{Varieties over $\bar{\mathbb{F}}_p$}
Varieties over finite fields enjoy some special properties which we will exploit.
For example, any numerically trivial divisor on a projective variety over $\bar{\mathbb{F}}_p$ is
torsion [\ref{Ke99}]. Another example is this:

\begin{thm}[{[\ref{ta12}, Theorem 0.2]}]\label{abd-dim2}
Let $X$ be a normal projective
surface over $\bar{\mathbb{F}}_p$.
Let $\Delta$ be an effective $\mathbb{Q}$-divisor on $X$.
If $K_X +\Delta$ is nef, then $K_X +\Delta$ is semi-ample.
\end{thm}

\subsection{Semi-positivity of direct images of pluri-canonical sheaves}

The following result is extracted from [\ref{pa12}, 1.5, 1.6, 1.7 and the paragraph below 1.7].
It holds in a more general form but this is all we need in this paper.

\begin{thm}[{[\ref{pa12}]}] \label{t-Patakfalvi-main}
Let $f\colon X\rightarrow Z$ be a surjective morphism from a normal projective variety
to a smooth projective curve over an algebraically closed field $k$.
Assume $K_X$ is $\Q$-Cartier and that general fibers are strongly $F$-regular.

$\bullet$ If $K_X$ is nef over $Z$ and $K_X$ is semi-ample on the generic fibre of $f$,
then $K_{X/Z}$ is nef.

$\bullet $If $K_{X}$ is ample over $Z$, then $f_*\mathcal{O}_X(mK_{X/Z})$ is a nef vector bundle for any sufficiently
divisible natural number $m$.
\end{thm}

We will apply the theorem only when $X$ is a $3$-fold and general fibres have canonical singularities.

\subsection{Varieties with elliptic fibrations}
For fibrations whose general fibres are elliptic curves, we can use a weak canonical bundle formula
which allows us to do induction.

\begin{thm}\label{eft}
Let $f\colon X \rightarrow Z$ be a contraction between smooth projective varieties
over an algebraically closed field $k$ such that the general fibres are
elliptic curves. Then $\kappa(K_{X/Z})\ge 0$.
\end{thm}
\begin{proof}
This follows from [\ref{cz13}, 3.2].\\
\end{proof}

\subsection{Nef divisors with Kodaira dimension one}

\begin{lem}\label{l-s-ample-kappa=1}
Let $X$ be a normal surface projective over a field $k$. Let $L$ be a nef $\Q$-divisor
with $\kappa(L)=1$. Then $L$ is semi-ample.
\end{lem}
\begin{proof}
Let $X\bir Z$ be the rational map defined by the linear system $|mL|$ for some
sufficiently divisible $m>0$. Then $\dim Z=1$.
We can replace $X$ with the normalization of the graph of $X\bir Z$ hence assume
$X\bir Z$ is a morphism. We can in addition assume $L\ge H\ge 0$ where
$H$ is the pullbck of some ample $\Q$-divisor on $Z$. Since $L$ is not
big, its support does not intersect the generic fibre of $X\to Z$.

Let $F$ be a fibre of $X\to Z$ which has a common component with $L$.
Let $a$ be the smallest rational number such that $L-aF\le 0$ near $F$. Then $L-aF$ has no common
component with $F$ otherwise there would be two components $C,D$ of $F$ such that
$C$ intersects $D$, $C$ is not a component of $L-aF$ but $D$ is a component of $L-aF$ which
implies $(L-aF)\cdot C<0$, a contradiction. These arguments show that $L$ is the pullback of
some $\Q$-divisor on $Z$ which is necessarily ample, hence $L$ is semi-ample.\\
\end{proof}

\subsection{Generically trivial divisors}
We recall a result of Kawamata adapted to char $p>0$.

\begin{lem}[{[\ref{bw14}]}]\label{l-linear-pullback}
Let $f\colon X\to Z$ be a contraction between normal projective varieties over
an algebraically closed field $k$
 and $L$ a nef$/Z$ $\Q$-divisor on $X$ such that $L|_F\sim_\Q 0$ where $F$ is the generic
fibre of $f$. Assume $\dim Z\le 3$ if $k$ has char $p>0$. Then there exist a diagram
$$
\xymatrix{
X'\ar[r]^\phi\ar[d]^{f'} & X\ar[d]^f\\
Z'\ar[r]^\psi & Z
}
$$
with $\phi,\psi$ projective birational, and an $\R$-Cartier divisor $D$ on $Z'$ such that
$\phi^* L\sim_\Q f'^*D$.
Moreover, if $Z$ is $\Q$-factorial, then we can take $X'=X$ and $Z'=Z$.
\end{lem}

For a proof see [\ref{bw14}, Lemma 5.6].

\subsection{Generic fibres}
Generic fibres often inherit the properties of the ambient space.

\begin{lem}\label{l-gen-fib}
Let $f\colon X\to Z$ be a dominant morphism of integral Noetherian schemes and let $F$ be its
generic fibre. Then the following statements hold:

$(1)$ $F$ is integral,

$(2)$ if $X$ is normal, then $F$ is also normal, and

$(3)$ if $X$ is regular, then $F$ is also regular.
\end{lem}
\begin{proof}
We can assume that $X$ and $Z$ are both affine, say $X=\Spec B$ and $Z=\Spec A$.
Let $K$ be the fraction field of $A$ and let $L$ be the fraction field of $B$.

(1) We need to show that $K\otimes_AB$ is integral. Assume not.
Then there is some non-zero $a\in A$ such that $A_a\otimes_AB$ is also not integral.
But $A_a\otimes_AB\simeq B_a$ and since $X\to Z$ is dominant, $A\subseteq B$ hence
$B_a$ is an integral domain. This is a contradiction. (2) and (3) follow from the fact that
if $y\in F$ is a point and $x$ its image in $X$, then the local ring $\mathcal{O}_y$
of $y$ on $F$ is isomorphic to the local ring $\mathcal{O}_x$ of $x$ on $X$.\\
\end{proof}

\subsection{Easy additivity of Kodaira dimensions}
The following result is well-known to experts [\ref{fu77} Propostion 1].

\begin{lem}\label{l-adtv-of-kdim}
Let $f\colon X\rightarrow Z$ be a contraction between normal varieties projective over a field $k$.
Let $D$ be an effective $\mathbb{Q}$-Cartier $\Q$-divisor on $X$ and $H$ a big $\Q$-Cartier
 $\mathbb{Q}$-divisor on $Z$. Then
$$
\kappa(D + f^*H) \geq \kappa(D|_F) + \dim Z
$$
where $F$ is the generic fibre of $f$.
\end{lem}
\begin{proof}
Since $D$ is effective, it is enough to prove the statement with $H$ replaced by any positive multiple and
$D$ replaced by $D+lf^*H$ for some $l>0$.
If $V\to X$ is a map, we denote the pullback of $D$ to $V$ by $D_V$ (similar notation for other divisors).
Let $m$ be a sufficiently divisible natural number and let $d=\dim_K H^0(mD_F)-1$ where $K$ is the function field of $Z$.
Let $S$ be the normalization of the image of $\phi_{mD_F}\colon F\bir \PP^d_K$ whose dimension is equal to $\kappa(D_F)$.
Moreover, $\phi_{mD_F}$ induces a (not unique) map $\psi\colon X\bir \PP^d_Z$ over $Z$ which restricts to $\phi_{mD_F}$.
Let $T$ be the normalization of the image of $\psi$.
Let $Y$ be the normalization of the graph of $X\bir \PP^d_Z$ and $G$ the generic fibre of $Y\to Z$. 
We have induced morphisms $Y\to T$, $G\to S$, and $ G\to F$.

Let $A$ on $\PP^d_Z$ be the pullback of a hyperplane via the projection $\PP^d_Z\to \PP^d_k$.
Perhaps after changing $D$ up to $\Q$-linear equivalence we can assume $mD_G\ge A_G$.
Thus replacing $D$ with $D+lf^*H$ for some $l$ we can assume $mD_Y\ge A_Y$. Therefore, we may replace
$X$ with $T$ and replace $D$ with $A_T$. But then the statement is trivial in this case
because we can assume $A+f^*H$ is ample.

\end{proof}

\subsection{Covering Theorem}

\begin{thm}[{[\ref{iit} Theorem 10.5]}]\label{ct}
Let $f\colon X \rightarrow Y$ be a proper surjective morphism between smooth complete varieties. 
If $D$ is a Cartier divisor on $Y$ and $E$ an effective $f$-exceptional divisor on $X$, then
$$\kappa(f^*D + E) = \kappa(D).$$
\end{thm}

Here by $f$-exceptional we mean: for any prime divisor $P$ on $Y$, there is a prime divisor $Q$ 
on $X$ mapping onto $P$ such that $Q$ is not a component of $E$. 

\vspace{0.3cm}

\section{Log surfaces over nonclosed fields}

In this section $k$ will denote a field which is not necessarily algebraically closed.
Shafarevich [\ref{shaf}] studied the minimal model theory of regular surfaces over nonclosed fields
and Dedekind rings (see also [\ref{Liu}]), and Manin [\ref{Manin}] and Iskovskikh [\ref{iskov}] treated the special case of rational surfaces. None of them seems to have discussed the abundance problem.
If $k$ is perfect (eg, when char $k=0$) or if the surface is smooth over $k$, then one can often reduce problems to the algebraically closed
case by passing to the algebraic closure. But our main point here is that we can actually prove many things
by working over $k$ rather than the algebraic closure when char $k>0$.

\subsection{Curves with negative canonical divisor}

As a preparation we collect some results about curves.

\begin{lem}\label{rc}
Let $X$ be a local complete intersection integral projective curve over a field $k$, and let 
$l = H^0(\mathcal{O}_X)$. Assume that $\deg_k K_X < 0$.
Then

$\rm (i)$ $\mathrm{Pic}^0(X) = 0$;

$\rm (ii)$ $X$ is a conic over $l$, and $\deg_l K_X = -2$;

$\rm (iii)$ if $X$ is normal and $\mathrm{char}~ k >2$, then $X_{\bar{l}} \cong \mathbb{P}_{\bar{l}}^1$.
\end{lem}
\begin{proof}
By assumption $\deg_k K_X<0$, hence $h^1(\mathcal{O}_X)=h^0(K_X)=0$ which implies $p_a(X) \leq 0$. 
Then  (i) and (ii) follow from [\ref{Liu}, Chapter 9 Proposition 3.16], and (iii) is [\ref{CTX}, Lemma 6.5].

\end{proof}

\subsection{Reduced boundary of dlt pairs}

\begin{lem}\label{l-dlt-bnd-normal}
Assume $(X,B)$ is a $\Q$-factorial dlt pair of dimension two  over a field $k$.
Then every irreducible component of $\rddown{B}$ is regular.
\end{lem}
\begin{proof}
Let $S$ be a component of $\rddown{B}$ and let $x\in S$ be a closed point. As $(X,S)$ is plt,
 $S$ is regular at $x$ by [\ref{Kollar-singularities}, 3.35].\\
\end{proof}

\begin{prop}\label{p-contraction}
Let $(X,B)$ be a $\Q$-factorial dlt pair of dimension two projective over a field $k$ where $B$ is a $\Q$-boundary.
Assume $S$ is a component of $\rddown{B}$ and $A$ is an  ample $\Q$-divisor such that

$\bullet$ $(K_X+B)\cdot S<0$,

$\bullet$ $(K_X+B+A)\cdot S=0$, and

$\bullet$ $K_X+B+A$ is nef and big.

Then $S$ can be contracted by a birational morphism $X\to Y$ and the resulting
pair $(Y,B_Y)$ is  $\Q$-factorial dlt. Moreover, $(K_X+B)\cdot S\ge -2$.
\end{prop}
\begin{proof}
By perturbing the coefficients of $B$ (i.e. by replacing $B$ with $B-P$ and replacing $A$ with $A+P$
for some appropriate $P$) we can assume $S=\rddown{B}$. Since $K_X+B+A$ is nef and big,
we can write $K_X+B+A\sim_\Q H+D$ where $H$ is ample and $D\ge 0$. Since $(H+D)\cdot S=0$,
$S$ is a component of $D$ and $S^2<0$.
Let $\epsilon>0$ be a small
rational number such that $A':=A+\epsilon S$ is ample. Then $S$ is the only curve on $X$ such that
$(K_X+B+A')\cdot S<0$. Let $t$ be the smallest real number such that $L:=K_X+B+A'+tA$ is nef.
We want to show $L$ is semi-ample and that $L\cdot S=0$.
If char $k=0$, the last sentence and the other claims of the proposition
can be reduced to the algebraically closed case by passing to the algebraic closure. So
we will assume char $k>0$.

By definition $L$ is nef and big but not ample hence the augmented base locus ${\bf{B_+}}(L)\neq \emptyset$.
Thus the exceptional set $\mathbb{E}(L)\neq \emptyset$  [\ref{CMM}], so there is a curve $C$ with $L\cdot C=0$
which implies $t$ is a rational number.
Actually $C=S$ by construction.

By Lemma \ref{l-dlt-bnd-normal}, $S$ is regular.
Then by adjunction (\ref{ss-adjunction}) we can write $K_S+B_S=(K_X+B)|_S$ where
$B_S\ge 0$. Since $\deg_k (K_S+B_S)<0$, we have $\deg_k K_S<0$.
This implies $\Pic^0(S)=0$ by Lemma \ref{rc}. Therefore $L|_S\sim_\Q 0$ which implies
that $L$ is semi-ample [\ref{Ke99}], so it defines a birational contraction $X\to Y$
contracting exactly $S$ so that $L_Y$ is ample where $L_Y$ is the pushdown of $L$.

The dlt property of $(Y,B_Y)$ is obvious once we show $Y$ is $\Q$-factorial where $B_Y$ is the
pushdown of $B$.
Let $R_Y$ be a prime divisor on $Y$ and $R$ its birational transform on $X$.
There is $s\ge 0$ such that $(R+sS)\cdot S=0$. Since $L$ is the pullback of an
ample divisor on $Y$, the divisor $M:=mL+R+sS$ is nef and big on $X$, and $\mathbb{E}(M)=S$ for any $m\gg 0$.
Moreover, $M|_S\sim_\Q 0$, so by [\ref{Ke99}, Theorem 0.2], $M$ is semi-ample, thus it is the pullback of some
ample divisor $M_Y$ on $Y$. But then $R_Y=M_Y-mL_Y$ is $\Q$-Cartier. This shows
$Y$ is $\Q$-factorial. Finally 
$$
(K_X+B)\cdot S=\deg_k(K_S+B_S)\ge \deg_k K_S=-2
$$
 
\end{proof}

\subsection{Base point freeness}

\begin{prop}\label{p-s-ample-big}
Let $(X,B)$ be a klt pair of dimension two projective over a field $k$ where
$B$ is a $\Q$-boundary. Assume $L$ is a nef and big $\Q$-divisor so that $L-(K_X+B)$ is nef.
Then $L$ is semi-ample.
\end{prop}
\begin{proof}
 If char $k=0$, we can  pass to the algebraic closure of $k$ in which case the theorem is well-known. So
we will assume char $k>0$.

Since $L$ is nef and big, we can write $L\sim_\Q H+D$ where $H$ is ample and $D\ge 0$.
Moreover, we can assume $\Supp D={\bf{B_+}}(L)$, hence $L|_{\Supp D}\equiv 0$ [\ref{CMM}]. By [\ref{Ke99}, Theorem 1.9],
there is a
birational morphism $X\to V$ to a proper algebraic space $V$ which contracts exactly $D$
and that $L\equiv 0/V$.

Let $\phi\colon W\to X$ be a log resolution of $(X,B+D)$. Let $\Delta_W$ be the sum of the birational
transform of $B_V$ plus the reduced exceptional divisor of $W\to V$ where $B_V$ is the
pushdown of $B$ on $V$. Let $R_W$ be an ample
divisor on $W$ and let $L_W$ be the pullback of $L$. Also let $G=L-(K_X+B)$ and $G_W$ be its pullback.
Fix $m\gg 0$ and let $t$ be the smallest
number such that
$$
N_W:=K_W+\Delta_W+G_W+tR_W
$$
is nef. Note that by construction, $K_W+\Delta_W+G_W=L_W+E_W$ where $E_W\ge 0$ and its support is equal to the
exceptional locus of $W\to V$.
Moreover, $N_W$ is nef and big but not ample, so by [\ref{CMM}], there is a curve $S$ with $N_W\cdot S=0$.
Since $(K_W+\Delta_W+G_W)\cdot S<0$, $E_W\cdot S<0$, hence $S$ is a component of $E_W$
which is contracted over $V$, so it is a component of $\Delta_W$.
In addition, $t$ is a rational number and $(K_W+\Delta_W)\cdot S<0$. Therefore, by Proposition \ref{p-contraction},
$S$ can be contracted by a birational morphism $W\to W'$ with an induced morphism $W'\to V$.
Continuing this process gives an LMMP on $K_W+\Delta_W$ over $V$.
It terminates with some model $Y$ on which $K_Y+\Delta_Y$ is nef$/V$.

Since $K_Y+\Delta_Y\equiv E_Y/V$, $E_Y$ is nef over $V$ and since $E_Y$ is exceptional over $V$,
we deduce $E_Y=0$ by the negativity lemma (which holds over arbitrary fields).
Therefore, $Y=V$ because $E_Y$ contains all the exceptional curves of $Y\to V$.
Thus $V$ is projective and $\Q$-factorial. Now $L_V$ is ample
and it pulls back to $L$, hence $L$ is semi-ample.\\
\end{proof}

\begin{prop}\label{p-s-ample-v=1}
Let $(X,B)$ be a klt pair of dimension two projective over a field $k$ where
$B$ is a $\Q$-boundary. Assume $L$ is a nef $\Q$-divisor so that $L-(K_X+B)$ is nef and big, and
$L$ is not numerically trivial.
Then $L$ is semi-ample.
\end{prop}
\begin{proof}
By Proposition \ref{p-s-ample-big}, we can assume $L$ is not big.  Moreover, replacing
$X$ with its minimal resolution we can assume $X$ is regular.
Let $G:=L-(K_X+B)$.
By the Riemann-Roch theorem for regular surfaces (see \ref{ss-intersection}), for any sufficiently divisible natural number
$m$ we have
$$
\mathcal{X}(mL)=\frac{1}{2}mL\cdot (mL-K_X)+\mathcal{X}(0)
$$
Since $G$ is big and $L$ is not numerically trivial, $L\cdot G>0$, and since
$$
mL-K_X\sim_\Q (m-1)L+B+G
$$
$L\cdot (mL-K_X)>0$, hence $\mathcal{X}(mL)$
is large when $m$ is large. This implies $h^0(mL)\ge 2$ for such $m$ because
$h^2(mL)=h^0(K_X-mL)=0$. Therefore, $L$ is semi-ample by Lemma \ref{l-s-ample-kappa=1}.\\
\end{proof}

\subsection{Running the LMMP}

\begin{proof}(of Theorem \ref{t-lmmp-nonclosed-surf})
\emph{Step 1}. Assume $K_X+B$ is pseudo-effective but
not nef. First suppose $X$ is $\Q$-factorial (we will see in the next step that this is
automatically satisfied).
Let $H$ be an ample divisor on $X$ and let $t$ be the smallest number such that $L=K_X+B+tH$ is nef.
Obviously $L$ is nef and big. Moreover, $t$ is rational which can be seen as in the proof of Proposition \ref{p-contraction}.
Although we can apply Proposition
\ref{p-s-ample-big} to deduce that $L$ is semi-ample and defines a contraction but
we want to modify the situation so that the contraction contracts only one curve.
Pick a curve $C$ such that $L\cdot C=0$. Let $\Delta=(1-\delta)B+\epsilon C$
for certain small rational numbers $\epsilon,\delta>0$ so that $(X,\Delta)$ is klt,
$(K_X+\Delta)\cdot C<0$, and $\delta B+tH$ is ample. Now let $t'$ be the smallest number such that
$L':=K_X+\Delta+\delta B+t'H$ is nef. Then $t<t'$ because $C^2<0$, so $L'$ is nef and big, and $\delta B+t'H$ is ample.
Note that $C$ is the only curve satisfying $L'\cdot C=0$.

Now by Proposition \ref{p-s-ample-big}, $L'$ is semi-ample and it defines a non-trivial birational
contraction $X\to Y$ contracting $C$ with $(K_X+B)\cdot C<0$. Let $R_Y$ be a prime divisor on $Y$ 
and $R$ its birational transform on $X$. Let $r$ be the number such that $(R+rC)\cdot C=0$. 
If $m>0$ is sufficiently large, then $L'':=mL'+R+rC$ is nef and big. Moreover, 
applying \ref{p-s-ample-big} to $L'+L''$ shows that $L''$ is the pullback of an 
ample divisor on $Y$, hence $R_Y$ is $\Q$-Cartier. 
Therefore, $Y$ is $\Q$-factorial and $(Y,B_Y)$ is dlt. Now replace $(X,B)$ with $(Y,B_Y)$
and repeat the argument.\\

\emph{Step 2.}
In this step we show that the dlt property of $(X,B)$ implies $X$ is $\Q$-factorial.
Since the pair is dlt, there is a log resolution $\phi\colon W\to X$ such that the log discrepancy
$a(D,X,B)>0$ for every curve $D$ contracted by $\phi$.
Let $\Gamma_W$ on $W$ be the sum of the birational
transform of $B$ and the reduced exceptional divisor of $W\to X$. Then $K_W+\Gamma_W=\phi^*(K_X+B)+E_W$
where $E_W\ge 0$ is contracted over $X$.
Let $H_W$ be the pullback of the ample divisor $H$. Fix $m\gg 0$. Then applying a similar procedure as
above we can run an LMMP on $K_W+\Gamma_W+mH_W$. Since $m\gg 0$, the curves $C$ contracted by the LMMP
intersect $H_W$ trivially, by Proposition \ref{p-contraction}, so such curves are contracted by $\phi$. In other words,
the LMMP is over $X$. The LMMP contracts $E_W$ so it ends with $X$ which means $X$ is $\Q$-factorial.
This and the previous step together prove the theorem when $K_X+B$ is pseudo-effective.\\

\emph{Step 3.}
From now on we assume $K_X+B$ is not pseudo-effective. If there is a curve $C$ such that
$(K_X+B)\cdot C<0$ and such that there is a birational morphism $X\to Y$ contracting exactly $C$,
then we replace $(X,B)$ with $(Y,B_Y)$. So we can assume there is no such $C$.

Pick an ample divisor $A$
and let $t$ be the smallest number such that $K_X+B+tA$ is pseudo-effective.
By the last paragraph $K_X+B+tA$ is nef: otherwise $K_X+B+t'A$ is not nef for a rational number
$t'>t$ sufficiently close to $t$, so we can run an LMMP on $K_X+B+t'A$ which is also an
LMMP on $K_X+B$ contracting some $C$, a contradiction.

If $\rho(X)=1$, then we already have a Mori fibre space. So assume $\rho(X)>1$.
Then there is another ample divisor $H$ such that $A$ is not numerically equivalent to
$hH$ for any number $h$. Let $s$ be the smallest number such that $K_X+B+sH$ is pseudo-effective.
 Arguing as above, $K_X+B+sH$ is nef.
By our choice of $A$ and $H$ both $K_X+B+tA$ and $K_X+B+sH$ cannot be numerically trivial at the same time.
We may assume $K_X+B+tA$ is not numerically trivial.\\

\emph{Step 4}.
In this step we assume $t$ is a rational number.
By Proposition \ref{p-s-ample-v=1}, $K_X+B+tA$ is semi-ample defining a contraction $f\colon X\to Z$ onto
$Z$ of dimension one. Assume there is a fibre $F$ of $f$ which is not irreducible. Let $C$ be a
component of $F$. Then $C^2<0$. We can find a $\Q$-boundary $\Delta$ such that $(X,\Delta)$ is klt,
$K_X+\Delta$ is pseudo-effective, and $(K_X+\Delta)\cdot C<0$. So we can contract $C$.
But since $(K_X+B)\cdot C<0$, this contradicts the first paragraph of Step 3.
Therefore, we can assume all the fibres of $f$ are irreducible. But this means $f$
is extremal and so $f$ is a Mori fibre space.\\

\emph{Step 5.}
Finally we show $t$ is indeed a rational number. Assume not. We derive a contradiction.
Let $L=K_X+B+tA$. For each sufficiently divisible natural number $m$, let $a_m$ be the number so that
$\rddown{mL}=mL-a_mA$. Since $t$ is not rational, there is an infinite set $\Pi$ of such $m$
so that the $a_m$ form a strictly decreasing sequence with $\lim_{m\in \Pi} a_m=0$. On the other hand,
for each $m\in\Pi$, let $a_m'$ be the number so that
$$
(mL-a_m'A)\cdot ({(m-1)L}+B+(t-a_m)A)
$$
$$
=(mL-a_m'A)\cdot (\rddown{mL}-K_X)=0
$$
Since $\lim_{m\in \Pi} a_m=0$ and $L^2=0$, we can see
$$
\lim_{m\in \Pi} a_m'=\lim_{m\in \Pi} \frac{mL\cdot ({(m-1)L}+B+(t-a_m)A)}{A\cdot ({(m-1)L}+B+(t-a_m)A)}=\frac{L\cdot (B+tA)}{A\cdot L}>0
$$
Thus we can assume
$$
(mL-a_mA)\cdot (\rddown{mL}-K_X)>\tau (m-1) A\cdot L
$$
for some $\tau>0$ independent of $m$.
Therefore,
$$
\mathcal{X}(\rddown{mL})=\frac{1}{2}(mL-a_mA)\cdot (\rddown{mL}-K_X)+\mathcal{X}(0)
$$
is large when $m\in \Pi$ is large. This in turn implies $h^0(\rddown{mL})$ is large for such $m$,
in particular, $mL\sim_\Q M$ for some $M\ge 0$. If $C$ is a component of $M$, then $L\cdot C=0$
which means $t$ is a rational number, a contradiction.\\
\end{proof}

\subsection{Mori fibre spaces}

\begin{prop}\label{p-Mfs}
{Let $(X,B)$} be a dlt pair of dimension two projective over a field $k$.
Assume $f\colon X\to Z$ is a Mori fibre structure for $(X,B)$ where $\dim Z=1$.
Then the geometric general fibres of $f$ are conics and if char $k>2$ they are
smooth rational curves.
In particular, if $F$ is a general fibre, then $(K_X+B)\cdot F\ge -2$.
\end{prop}
\begin{proof}
Let $F$ be the generic fibre of $f$ which is a regular curve by Lemma \ref{l-gen-fib}.
Since $-(K_X+B)$ is ample over $Z$,
$-K_F$ is ample. On the other hand, since $f$ is a contraction,
$H^0(\mathcal{O}_F)=K$ where $K$ is the function field of $Z$.
The assertions follow from Lemma \ref{rc} straightforwardly.\\
\end{proof}

\subsection{Curves of canonical type}
Let $X$ be a regular surface projective over a field $k$.
A connected divisor $D=\sum_1^r d_iD_i\ge 0$ is called a curve of canonical type if $D|_D\equiv 0$ and $K_X|_D\equiv 0$.
It is called indecomposable if there is no prime number dividing all the $d_i$.
The following result was proved by Mumford [\ref{Mum}, page 332]. Although he assumes the ground field to be
algebraically closed but his proof works for arbitrary fields. We give the proof for convenience
(see also [\ref{Bad}, Theorem 7.8]).

\begin{prop}\label{p-can-curve-lin-0}
Let $D$ be an indecomposable curve of canonical type.
Let $L$ be a Cartier divisor on $D$ such that $L\equiv 0$.
If $h^0({L})\neq 0$, then $L\sim 0$.
\end{prop}
\begin{proof}
Assume $\alpha\in H^0(L)$ is nonzero. Then $\alpha|_{D_i}$ is either nowhere
vanishing or everywhere vanishing because $L|_{D_i}\equiv 0$. Since $D$ is connected, either $\alpha$ is
nowhere vanishing on $D$ or $\alpha|_{\Supp D}=0$. The former implies $L$ is generated by global
sections which in turn implies $L\sim 0$.
So it is enough to treat the latter. Let $n_i$ be the order of vanishing
of $\alpha$ along $D_i$. Let $N=\sum n_iD_i$. We want to show $N=D$.

Assume $n_i<d_i$, say for $i=1$. Consider the exact sequence
$$
0\to \mathcal{O}_{D_1}(L-n_1D_1) \to \mathcal{O}_{(n_1+1)D_1}(L) \to \mathcal{O}_{n_1D_1}(L)\to 0
$$
Since $\alpha|_{n_1D_1}=0$ by definition of $n_1$, the section $\alpha|_{(n_1+1)D_1}$
is the image of a section $\beta$ of
$(L-n_1D_1)|_{D_1}$. If $P$ is the zero divisor of $\beta$, then
a local computation of intersection numbers shows that $P\ge (N-n_1D_1)|_{D_1}$.
More precisely, let $v\in D_1$ be a closed point, let $R=\mathcal{O}_{X,v}$, and let $f_i$ be a local equation of $D_i$
near $v$. Then locally considering $\alpha$ as an element of $\frac{R}{\langle f_1^{d_1}\cdots f_r^{d_r}\rangle}$,
it is easy to see that $\alpha$ is represented by $\lambda f_1^{n_1}\cdots f_r^{n_r}$ for some $\lambda\in R$, and that
$\beta$ is represented by $\lambda f_2^{n_2}\cdots f_r^{n_r}$ which gives the equation of $P$ near $v$.
Therefore, from
$$
{\rm length}_{\frac{R}{\langle f_1\rangle}} \frac{R}{\langle f_1,\lambda f_2^{n_2}\cdots f_r^{n_r}\rangle}\ge {\rm length}_R \frac{R}{\langle f_1,f_2^{n_2}\cdots f_r^{n_r}\rangle}
$$
we deduce that locally near $v$ we have $P\ge (N-n_1D_1)|_{D_1}$ because the left hand side  of the displayed formula
is the coefficient of
$v$ in $P$ and the right hand side is nothing but the local intersection number $(N-n_1D_1)\cdot D_1$ at $v$ which is in turn
equal to the coefficient of $v$ in $(N-n_1D_1)|_{D_1}$. As $P\sim (L-n_1D_1)|_{D_1}$, we deduce that
$\deg N|_{D_1}\le 0$. Thus $N$ is anti-nef. Letting $a$ be the smallest number such that $aN-D\ge 0$ and
taking intersection numbers one shows $D=aN$. Since the $d_i$ have no common factor, $N=D$. So $\alpha=0$,
a contradiction.\\
\end{proof}

\begin{prop}
Let $D$ be an indecomposable curve of canonical type.
Then the arithmetic genus $p_a(D)=1$ and $K_D\sim 0$.
\end{prop}
\begin{proof}
By definition of curves of canonical type $K_D=(K_X+D)|_D\equiv 0$.
By [\ref{Liu}, Chapter 7, Corollary 3.31], $0=\deg_k K_D=2(p_a(D)-1)$.
Thus $p_a(D)=1$ which means $\mathcal{X}(\mathcal{O}_D)=0$, hence $h^1(\mathcal{O}_D)=h^0(\mathcal{O}_D)>0$.
So by duality $h^0(K_D)=h^1(\mathcal{O}_D)>0$ which implies $K_D\sim 0$ by Proposition \ref{p-can-curve-lin-0}.\\
\end{proof}

\begin{prop}\label{p-s-ample-can-type}
Assume char $k>0$. Let $D$ be an indecomposable curve of canonical type such that $D|_D$ is torsion. Then $D$ is semi-ample on $X$.
\end{prop}
\begin{proof}
Let $r$ be the order of $D|_D$ in $\Pic(D)$. First we want to show $rD|_{rD}\sim 0$.
This is trivially true if $r=1$, so assume $r>1$.
Assume we already know $rD|_{lD}\sim 0$ for some $0<l<r$. Consider the exact sequence
$$
0\to \mathcal{O}_{D}(rD-lD) \to \mathcal{O}_{(l+1)D}(rD) \to \mathcal{O}_{lD}(rD)\to 0
$$
Now $h^0(\mathcal{O}_{D}(rD-lD))=0$ by Proposition \ref{p-can-curve-lin-0}, and since $\mathcal{X}(\mathcal{O}_D)=0$,
by Riemann-Roch we get
$$
\mathcal{X}(\mathcal{O}_D(rD-lD))=\deg_k (rD-lD)|_D+\mathcal{X}(\mathcal{O}_D)=0
$$
which implies $h^1(\mathcal{O}_{D}(rD-lD))=0$. So any nowhere vanishing section of
$\mathcal{O}_{lD}(rD)$ lifts to $\mathcal{O}_{(l+1)D}(rD)$ which shows $rD|_{(l+1)D}\sim 0$.
Inductively one shows $rD|_{rD}\sim 0$.
Finally applying [\ref{tot09}, Lemma 4.1], we deduce $D$ is semi-ample.\\
\end{proof}

\subsection{Abundance}

\begin{proof}(of Theorem \ref{t-abund-nonclosed-surf})
We can assume $K_X+B$ is not big by Proposition \ref{p-s-ample-big}.
Replacing $X$ with its minimal resolution we can assume $X$ is regular.
By assumption $m(K_X+B)\sim M$ for some $m>0$ and $M\ge 0$. Let $n$ be a sufficiently large
natural number. We can run an LMMP on $K_X+n M$ because
$$
K_X+n M\sim (1+nm)(K_X+\frac{nm}{1+nm}B)
$$
and because $(X,\frac{nm}{1+nm}B)$ is klt.
Moreover, we claim that $M$ is numerically trivial on each step of the LMMP and
the nefness of $M$ is preserved in the process. Indeed assume the first step  of the LMMP is a
birational contraction $X\to Z$ contracting a curve $E$. Then $\deg_k K_E=(K_X+E)\cdot E<0$, hence
by Lemma \ref{rc}, if setting $l=H^0(\mathcal{O}_E)$ then
$$
-2=\deg_lK_E=\deg_l(K_X+E)|_E
$$
which implies $\deg_l K_X|_E=-1$. Thus from $\deg_l(K_X+nM)|_E<0$ we deduce $M\cdot E=0$.
On the other hand, if $X\to Z$ is a Mori fibre space, then we stop the LMMP and in this
case $M\equiv 0/Z$ by Proposition \ref{p-Mfs} and calculations similar to those above.
Applying this argument to every step of the LMMP proves the claim.
Note that the regularity of $X$ is also preserved by the LMMP because the LMMP is a $K_X$-LMMP hence
$X$ remains with terminal singularities which implies regularity.

Replacing $X$ with the end product of the LMMP we can assume either $K_X+n M$ is nef or
that there is a Mori fibre structure $X\to Z$ so that $M\equiv 0/Z$.
First assume $K_X+n M$ is nef.
There is a divisor $0\le D=\sum d_i D_i\le M$
such that $D$ is connected, the $d_i$ have no common prime factor, and $M=aD$ in a
neighbourhood of $D$ for some number $a$. In particular, $D$ is nef. We show $D$ is an indecomposable curve of canonical
type. It is enough to show $K_X|_D\equiv 0$ because $M$ not being big implies $D|_D\equiv 0$.
Since $M$ is not big and since $m(K_X+B)\sim M$, we deduce $K_X+n' M$ is not big for any $n'$.
Therefore,  $(K_X+nM+M)^2=0$ from which we deduce $(K_X+nM)\cdot M=0$, hence
$(K_X+n M)|_D\equiv 0$, so $K_X|_D\equiv 0$.

In order to apply Proposition \ref{p-s-ample-can-type} we need to show $D|_D$ is torsion.
By construction, $B|_D\equiv 0$ which implies $B=bD$ in some neighbourhood
of $D$ because $D$ is connected, where $b<1$ is a rational number. Taking $m$ so that $mb\in \Z$ we get
$$
0\sim mK_D=m(K_X+D)|_D=m(K_X+B+D)|_D-mB|_D\sim (a+m)D|_D-mbD|_D
$$
which implies $D|_D$ is torsion because $a+m-mb>0$.
Therefore, $D$ is semi-ample, hence $\kappa(M)=1$ which implies $M$ is semi-ample by Lemma \ref{l-s-ample-kappa=1}.

Now assume we have a Mori fibre structure $X\to Z$ with $M\equiv 0/Z$. If $F$ is the generic fibre,
then $M|_F\sim 0$. This implies $M$ is the pullback of
some effective divisor $N$ on $Z$. Either $N$ is ample or $N=0$,  hence in any case $M$ is semi-ample.
\end{proof}

\begin{rem}\label{r-nonclosed-surf}
Here we explain what we need from this section for the proof of Theorem \ref{t-main}.
We will need Theorem \ref{t-abund-nonclosed-surf} for the proof of Theorem \ref{t-3d-rel-mmodel}.
In turn we use Theorem \ref{t-3d-rel-mmodel} in the proofs of Corollary \ref{c-k=1-it-fib} and
Proposition \ref{k2} (steps 1 and 5)
in two situations: when (1) $F$ is smooth and when (2) $\kappa(K_F+B_F:=(K_X+B)|_F)=0$ and there is a
surjective map $F\to C$ onto an elliptic curve defined over the function field $K$ of $Z$
and such that $K_F+B_F$ is big over $C$. In each case it is enough to know that $K_F+B_F$
is semi-ample.
In case (1), we can pass to the algebraic closure $\bar K$ and deduce that $K_F+B_F$ is
semi-ample. In case (2), $m(K_F+B_F)\sim M_F\ge 0$ for some $m>0$, and using the map $F\to C$
it is relatively easy to show $M_F=0$: if not then each connected component of $M_F$ is irreducible;
let $D$ be the reduction of such a component; then there is one such $D$ which maps onto
$C$ and one can show that $D$ is an elliptic curve with $K_F\cdot D=0$ and $D|_D$ torsion;
one then applies Proposition \ref{p-s-ample-can-type} to deduce that $D$ is semi-ample,
hence $\kappa(M_F)\ge 1$, a contradiction.\\
\end{rem}


\section{Relative good minimal models of $3$-folds}

\begin{proof}(of Theorem \ref{t-3d-rel-mmodel})
By [\ref{bi13}], $(X,B)$ has a log minimal model over $Z$. Replacing $X$ with the minimal model, we can assume
$K_X+B$ is nef$/Z$. Let $F$ be the generic fibre of $X\to Z$ and let $K_F+B_F=(K_X+B)|_F$. Then $(F,B_F)$
is klt and $K_F+B_F$ is nef with $\kappa(K_F+B_F)\ge 0$. By Theorem \ref{t-abund-nonclosed-surf},
$K_F+B_F$ is semi-ample.

If $\kappa(K_F+B_F) = 0$, then $K_F+B_F\sim_\Q 0$, hence $K_X+B\sim_\Q 0/Z$, by Lemma \ref{l-linear-pullback}.
On the other hand, if $\kappa(K_F+B_F) = 2$, then $K_X+B$ is big$/Z$, hence it is semi-ample over $Z$ by
Theorem \ref{t-mmp-3fold}.
So we will assume $\kappa(K_F+B_F) = 1$.

Since  $K_F+B_F$ is semi-ample and $\kappa(K_F+B_F) = 1$, there is a diagram
$$\centerline{\xymatrix{
& &Y \ar[ld]_\phi \ar[rd]^g  & \\
&X \ar[rd]_f  &   &S \ar[ld]^h  \\
& &Z &
}}$$
where $\phi$ is birational, $S$ is a smooth projective surface, and $\phi^*(K_X+B)|_G\sim_\Q 0$
on the generic fibre $G$ of $g$. By Lemma \ref{l-linear-pullback}, we can actually assume
$\phi^*(K_X+B)\sim_\Q 0/S$. So $\phi^*(K_X+B)\sim_\Q g^* D$ for some $\Q$-Cartier divisor
$D$ on $S$.
On the other hand, let $H$ be an ample divisor on $Z$. Then since $D$ is nef and big over $Z$,
 $D+nh^*H$ is nef and big
for any $n\gg 0$. Since we are working over $\bar{\mathbb{F}}_p$, $D+nh^*H$ is semi-ample (this follows from [\ref{Ke99}])
which implies $D$ is semi-ample over $Z$ from which we deduce $K_X+B$ is semi-ample over $Z$.\\
\end{proof}

\begin{cor}\label{c-k=1-it-fib}
Let $W$ be a smooth projective $3$-fold over an algebraically closed field $k$ of char $p>5$.
Assume $\kappa(K_W)=1$ and that $g\colon W\to C$ is the
Iitaka fibration. In addition assume $X$ is a minimal model of $W$.
Then the induced map $X\bir C$ is a morphism and $K_X\sim_\Q 0/C$.
\end{cor}
\begin{proof}
Let $Y$ be minimal model of $W$ over $C$ which comes with a morphism $r\colon Y\to C$.
Let $R$ be the generic fibre of $r$. We show that $R$ is a regular surface.
Since $W$ is smooth, $Y$ has terminal singularities. So $R$ also has terminal singularities
because if $\phi\colon V\to Y$ is a resolution, then it induces a birational morphism $\psi\colon S\to R$, where
$S$ is the generic fibre of $V\to C$, such that $K_S=\psi^*K_R+E$
where $E\ge 0$ is exceptional over $R$. Therefore, the minimal resolution of $R$ is isomorphic to
$R$, hence $R$ is regular.

Since $\kappa(K_W)=1$ and since $g$ is the Iitaka fibration of $K_W$, $\kappa(K_R)=0$. Moreover,
as $K_Y$ is nef over $C$, $K_R$ is nef too.
 Therefore, by Theorem \ref{t-abund-nonclosed-surf}, $K_R\sim_\Q 0$. This implies $K_Y\sim_\Q 0/C$
 as $K_Y$ is nef over $C$, by Lemma \ref{l-linear-pullback}. Thus $K_Y$ is the pullback of an ample divisor on $C$.
In particular, this means $Y$ is a minimal model of $W$ globally, not just over $C$.

Now let $X$ be any minimal model of $W$. Then $X$ and $Y$ are isomorphic in codimension one.
Moreover, the induced map $X\bir C$ is a morphism because $C$ is the canonical model of both $X$
and $Y$. Therefore, $K_X\sim_\Q 0/C$ as claimed.\\
\end{proof}

\section{Kodaira dimensions}

In this section we prove some results on Kodaira dimensions which will be used in the proof of
Theorem \ref{t-main}.

\begin{prop} \label{gk}
Let $f\colon X\rightarrow Z$ be a contraction from a smooth projective variety onto a smooth projective curve
over an algebraically closed field $k$ of char $p>0$. Assume there is an integer $m>1$ such that $f_*\omega_{X/Z}^m$ 
is non-zero and nef. If either
\begin{itemize}
\item[(1)]
$g(Z) >1$; or
\item[(2)]
$g(Z) = 1$ and $\deg f_*\omega_{X/Z}^m>0$,
\end{itemize}
then $\kappa(K_X) \geq \kappa(K_F) + 1$.
\end{prop}

\begin{proof}
(1) 
Since $f_*\omega_{X/Z}^m$ is non-zero, for each integer $l>0$ there is a non-zero map 
 $S^l(f_*\omega_{X/Z}^m)\to f_*\omega_{X/Z}^{lm}$. 
Since $f_*\omega_{X/Z}^m$ is nef hence weakly positive, for sufficiently big $l$, $S^l(f_*\omega_{X/Z}^m\otimes \mathcal{O}_Z(P))$ is globally generated, so $f_*\omega_{X/Z}^{lm}\otimes \mathcal{O}_Z(lP))$ has a non-zero global section which implies
$l(mK_{X/Z} + f^*P)$ is linearly equivalent to an effective divisor where $P$ is a point on $Z$.

As $g(Z) > 1$, we can write $K_Z \sim P + N$ where $N$ is an effective divisor. Then
$$
lmK_X \sim  l(mK_{X/Z} + f^*P) + (m-1)lf^*P + lmN
$$
Applying Lemma \ref{l-adtv-of-kdim} we are done in this case.

(2)
Let $r=\text{rank }f_*\omega_{X/Z}^m$.
Consider the base change
$$
\xymatrix{X'=X\times_ZZ'\ar[d]_{f'}\ar[rr]^g&& X\ar[d]^f\\
Z '\ar[rr]^{\pi}&& Z }$$
where $\pi$ is an \'etale map of degree $d>r$ and $Z'$ is integral.
Note that such a $\pi$ exists by considering the algebraic fundamental
group of $Z$. Moreover, $Z'$ is again an elliptic curve.
Since $f$ and $\pi$ are both flat, we have
$\omega_{X'/Z'}=g^*\omega_{X/Z}$ and $f'_*\omega_{X'/Z'}^m\cong \pi^*f_*\omega_{X/Z}^m$.
By Riemann-Roch for vector bundles over a curve, we have
$$
\begin{array}{rcl}&&h^0(f'_*\omega_{X'/Z'}^m\otimes \mathcal{O}_{Z'}(-P'))\\
&\geq& \deg (f'_*\omega_{X'/Z'}^m\otimes \mathcal{O}_{Z'}(-P'))\\
&=&\deg f'_*\omega_{X'/Z'}^m-r\\
&=&d\deg f_*\omega_{X/Z}^m-r>0,
\end{array}$$
where $P'\in Z'$ is a closed point. So $mK_{X'} = mK_{X'/Z'}\sim f'^*P'+E$ for some effective divisor $E$ on $X'$. Then
$$
\kappa(K_X)=\kappa(K_{X'})= \kappa(K_{X'} +f'^*P')\ge \kappa(K_{F'})+1=\kappa(K_F)+1
$$
by Lemma \ref{l-adtv-of-kdim} where $F'$ is the generic fibre of $f'$.\\
\end{proof}

\begin{prop}\label{k0}
Let $f\colon X\rightarrow Z$ be a contraction from a normal projective variety to an elliptic curve over $\bar{\mathbb{F}}_p$.
Assume that $f_*\omega_{X/Z}^l$ is a non-zero nef vector bundle for some $l>0$. Then $\kappa(K_X) \geq 0$.
\end{prop}
\begin{proof}
If $\deg f_*\omega_{X/Z}^l >0$, then
$$
h^0(\omega_X^l)=h^0(f_*\omega_{X/Z}^l) \geq \chi(f_*\omega_{X/Z}^l) >0
$$
So we can assume $\deg f_*\omega_{X/Z}^l  = 0$. The vector bundle $f_*\omega_{X/Z}^l$ can
be decomposed into a direct sum $\bigoplus_iV_i \otimes L_i$ where $V_i$ are nef indecomposable
vector bundles with $h^0(V_i) = 1$ and $L_i$ are line bundles with $\deg L_i = 0$. Since we work over $\bar{\mathbb{F}}_p$,
$L_1$ is torsion, say of order $n$. We have an \'{e}tale cover $\pi\colon Z' \rightarrow Z$
induced by the relation $L_1^n \simeq \mathcal{O}_Z$. Consider the base change
$f'\colon X' = X \times_Z Z' \rightarrow Z'$. Then
$f'_*\omega_{X'/Z'}^l \cong \pi^* f_*\omega_{X/Z}^l$.
So $f'_*\omega_{X'/Z'}^l$ contains $\pi^*V_1 \otimes \pi^*L_1 \cong \pi^*V_1$,
hence
$$
h^0(\omega_{X'}) = h^0(f'_*\omega_{X'/Z'}^l) \geq h^0(\pi^*V_1) = 1
$$
So $\kappa(K_X) = \kappa(K_{X'}) \geq 0$.\\
\end{proof}

\begin{prop}\label{k2}
Let $f\colon X\rightarrow Z$ be a contraction from  a projective $3$-fold with $\Q$-factorial
terminal singularities to an elliptic curve
over $\bar{\mathbb{F}}_p$ with $p>5$.
Assume that $K_{X}$ is big over $Z$ and that the generic fiber of $f$ is smooth. Then $\kappa(X) \geq 2$.
\end{prop}
\begin{proof}
\emph{Step 1.}
By Theorem \ref{t-mmp-3fold}, there is a minimal model $Y$ for $X$ over $Z$.
Let $F$ and $G$ be the generic fibres of $X\to Z$ and $Y\to Z$ respectively, and let
$\pi\colon F\to G$ be the induced morphism. We want to show $G$ is smooth over $K$ where $K$
is the function field of $Z$.
We denote $F_{\bar K}\to G_{\bar K}$ by $\bar\pi$,
$G_{\bar K}\to G$ by $\mu$, and $F_{\bar K}\to F$ by $\rho$. Since $\phi\colon X\bir Y$ is a contraction near $F$,
$\pi$ is also a contraction, so $\pi_*\mathcal{O}_F=\mathcal{O}_G$. Moreover, since $\mu$ is flat,
we have
$$
\bar{\pi}_*\mathcal{O}_{F_{\bar K}}=\bar{\pi}_*\rho^*\mathcal{O}_{F}={\mu}^*\pi_*\mathcal{O}_{F}
=\mu^*\mathcal{O}_{G}=\mathcal{O}_{G_{\bar K}}
$$
Therefore, $\bar \pi$ is a contraction which in particular implies $G_{\bar K}$ is normal.

On the other hand, near $F$, $K_X=\phi^*K_Y+E$ for some effective divisor $E$
whose support near $F$ is equal to the union of all the exceptional$/Y$ divisors near $F$.
Thus $K_{F_{\bar K}}=\bar{\pi}^*K_{G_{\bar K}}+E_{\bar K}$ where $E_{\bar K}\ge 0$
whose support contains all the exceptional curves of $\bar{\pi}$. We deduce $G_{\bar K}$
has terminal singularities, hence it is smooth. Therefore, $G$ is smooth over $K$.

 Replacing $X$ with $Y$, we can assume $K_X$ is nef$/Z$.
Since $Z$ is an elliptic curve, $K_X$ is actually globally nef by the cone theorem [\ref{bw14}, Theorem 1.1].\\

\emph{Step 2.}
Let $X'$ be the canonical model of $X$ over $Z$ which exists by Theorem \ref{t-mmp-3fold}.
So $K_{X'}$ is ample over $Z$.
The general fibres of $f'\colon X'\to Z$ are normal because they are regular in codimension one 
and they are Cohen-Macaulay. Thus they have canonical singularities, hence are strongly $F$-regular. Therefore, by Theorem \ref{t-Patakfalvi-main},  $f'_*\omega_{X'}^m=f_*\omega_{X}^m$
is a nef vector bundle for any sufficiently divisible $m>0$.

If $\nu(K_X) = 3$, then $\kappa(K_X)=3$, so there is nothing to prove. On the other hand,
if $f_*\omega_X^m$ is nef with $\deg f_*\omega_X^m > 0$ for some $m>0$, 
then we are done by applying Proposition \ref{gk} to a resolution of $X$.
 So in the following we assume $\nu(K_X) = 2$ and that $\deg f_*\omega_X^m = 0$ for any sufficiently divisible positive integer $m$.\\

\emph{Step 3.}
To ease notation, from here until the end of Step 4 we replace $X$ with $X'$.
Applying Proposition \ref{k0} and Step 2,
we find a positive integer $l$ and a divisor $M\ge 0$ such that $lK_X \sim M$
and that $f_*\omega_X^{kl}$ is nef for any $k\geq 1$.
We prove that $K_X|_M$ is semi-ample.

Let $T$ be a component of $M$. Then $T$ is horizontal$/Z$ otherwise $M$ would be big by considering $M^3$.
Take the normalization $S\rightarrow T$, and let ${C} \subset S$ be the reduction of the conductor. Write  $M = nT + T'$ where $T$ is not a component of $T'$. By adjunction [\ref{Ke99}, 5.3], we have
$$
(K_X + \frac{M}{n})|_S \sim_{\mathbb{Q}} K_S + {C} + {D}
$$
where ${D}$ is a canonically defined effective $\Q$-divisor and $|_S$ means pullback to $S$.
Then $K_X|_S$ is semi-ample on $S$ by Theorem \ref{abd-dim2}.
We want to argue that semi-ampleness of $K_X|_S$ implies semi-ampleness of $K_X|_T$.

Since $K_X$ is nef and $K_X^3 = 0$, we have $K_X^2\cdot T = 0$, i.e., $(K_X|_T)^2 = 0$, thus $(K_X|_S)^2 = 0$. Then since $K_X|_T$ is ample over $Z$, we get $\nu(K_X|_S) = 1$. We denote by $h\colon S \rightarrow V$ the map  associated  to $K_X|_S$ and denote by $H$ a general fiber which has genus $g(H) \geq 1$ because it dominates $Z$. As $K_X|_S \cdot H = 0$, we have
$(K_S + {C} + {D})\cdot H = 0$, hence
$$
0\le \deg K_H = (K_S + H) \cdot H = K_S \cdot H = -({C} + {D}) \cdot H\leq 0
$$
Therefore, ${C} \cdot H =  {D}\cdot H = 0$, and $H$ is smooth with genus $g(H) = 1$.
Applying [\ref{Ke99}, Corollary 2.15], we conclude that $K_X|_T$ is semi-ample, and the associated map $\bar{h}\colon T \rightarrow \bar{V}$ is an elliptic fibration. In particular, this means that no component of $T'$ intersects the general fibres of
$\bar h$.

Let $R\le M$ be a reduced divisor and assume that $K_X|_R$ is semi-ample. If $R=\Supp M$, then we are done by [\ref{Ke99}, Lemma 1.4].
If not, pick a component $T$ of $M$ which is not a component of $R$. As noted above, $K_X|_T$ is semi-ample defining a
contraction $\bar{h}\colon T \rightarrow \bar{V}$ whose general fibres do not intersect any component of $R$.
Applying [\ref{Ke99}, Corollary 2.12], we deduce that $K_X|_{T\cup R}$ is semi-ample.
Inductively, we extend $R$ to the support of $M$.\\

\emph{Step 4.}
In this step we prove $\kappa(K_X) \geq 1$.
Consider the following exact sequence
\begin{equation}\label{sq3}
 0 \rightarrow \mathcal{O}_X((k-1)M) \rightarrow \mathcal{O}_X(kM) \rightarrow \mathcal{O}_M(kM) \rightarrow 0
 \end{equation}
For $k\geq 2$, by assumptions in Step 2, both $f_*\mathcal{O}_X(kM)$ and $f_*\mathcal{O}_X((k-1)M) $ are nef vector bundles with
$$
\deg f_*\mathcal{O}_X(kM) = \deg f_*\mathcal{O}_X((k-1)M) = 0
$$
If
$$
h^0(f_*\mathcal{O}_X(kM)) = h^0(f_*\mathcal{O}_X((k-1)M)) = 1
$$
for all $k \geq2$, then $ h^1(f_*\mathcal{O}_X((k-1)M)) = 1$ for such $k$ by Riemann-Roch, and by taking cohomology of the exact sequence (\ref{sq3}), we conclude that $h^0(\mathcal{O}_M(kM)) \leq 1$. However, this contradicts semi-ampleness of $M|_M$ and the property $\nu(M|_M) \geq 1$.
Therefore, $\kappa(K_X) \geq 1$.\\

\emph{Step 5.}
Assume $\kappa(K_X)=1$. We will derive a contradiction.
Let $W\to X$ be a resolution so that the Iitaka fibration $W\to C$ is a morphism.
By Corollary \ref{c-k=1-it-fib}, the induced map $X\bir C$ is a morphism and $K_X\sim_\Q 0/C$.
In particular, $\nu(K_X)=1$ which contradicts the assumption $\nu(K_X)=2$.\\
\end{proof}

\section{Proof of Theorem \ref{t-main}}

\begin{proof}(of Theorem \ref{t-main})
We can assume $\kappa(K_Z) \geq 0$ and $\kappa(K_F) \geq 0$.
As pointed out in the introduction $C_{3,2}$ follows from [\ref{cz13}], so we will assume $n=3$ and $m=1$.
Replacing $X$ with a minimal model over $Z$, we can assume $K_X$ is nef$/Z$.
Of course $X$ may not be smooth any more but it has $\mathbb{Q}$-factorial terminal singularities.
The generic fibre stays smooth by the arguments in Step 1 of the proof of Theorem \ref{k2}.

If $\kappa(K_F) = 0$, then by Theorem \ref{t-3d-rel-mmodel} and Lemma \ref{l-linear-pullback}, 
$K_{X/Z} \sim_{\mathbb{Q}} f^*M$ for some $\Q$-divisor $M$. Moreover,
by Theorem \ref{t-Patakfalvi-main}, $K_{X/Z}$ is nef, hence $\deg M \geq 0$ which implies
$\kappa(M) \geq 0$ as we are working over $\bar{\mathbb{F}}_p$.
Thus $\kappa(K_{X/Z}) \geq 0$ and
$$
\kappa(K_X) = \kappa(K_{X/Z} + f^*K_Z) \geq \kappa(K_{Z})
$$

If $\kappa(K_F) = 1$, by Theorem \ref{t-3d-rel-mmodel}, there is a commutative diagram
$$\centerline{\xymatrix{
&X \ar[rd]_f\ar[rr]^g  &   &Y \ar[ld]^h  \\
& &Z &
}}$$
such that $g$ is an elliptic fibration (as $p>5$) and $K_X \sim_{\mathbb{Q}} g^*N$ for some $\Q$-divisor $N$.
By the cone theorem [\ref{bw14}, Theorem 1.1], $K_X$ is nef, hence $N$ is nef too.
On the other hand, by Theorem \ref{eft} applied to appropriate resolutions of 
$X$ and $Y$ we get $N \sim_{\mathbb{Q}} K_Y+ \Delta$  where $\Delta$ is effective.
Applying Theorem \ref{abd-dim2}, we can conclude that $K_Y+\Delta$ is semi-ample. So $K_X$ is semi-ample on $X$. Then since $K_{X/Z}$ is nef (Theorem \ref{t-Patakfalvi-main}), we have
$$
\kappa(K_X)=\nu(K_X) \geq \kappa(K_Z) + 1
$$
where $\nu(K_X)$ denotes the numerical dimension of $K_X$.

Finally assume $\kappa(K_F) = 2$. If $\kappa(K_Z)=0$ we apply Proposition \ref{k2}.
But if $\kappa(K_Z)=1$, we replace $X$ with its canonical model over $Z$ so that $K_X$ is ample
over $Z$, and we use Theorem \ref{t-Patakfalvi-main} to deduce $f_*\omega_{X/Z}^m$ is nef
for sufficiently divisible $m>0$ as in the proof of \ref{k2}; next we apply Proposition \ref{gk} to a resolution of $X$.\\
\end{proof}

\begin{proof}(of Corollary \ref{c-main})
By [\ref{Bad}, Corollary 7.3] $f$ has integral generic geometric fibre. Let $F$ be the generic fibre of $f$, let $F_1 = F\times_{K(Z)}K(Z)^{\frac{1}{p^{\infty}}}$, and let $\tilde{F}_1 \rightarrow F_1$ be a desingularization. Since $K(Z)^{\frac{1}{p^{\infty}}}$ is perfect, $\tilde{F}_1$ is smooth over $K(Z)^{\frac{1}{p^{\infty}}}$ by [\ref{Liu}, Chapter 4 Corollary 3.33]. Therefore, there exists a natural number $e$ such that $\tilde{F}_1$ can be descent to $\tilde{F}_2$, which is a desingularization of $F_2 = F\times_{K(Z)}K(Z)^{\frac{1}{p^e}}$ and smooth over $K(Z)^{\frac{1}{p^e}}$.

Denote by $F^e: Z' \rightarrow Z$ the $e$-th Frobenius iteration. We have the following commutative diagram
$$\xymatrix{X'\ar[drr]_{f'}\ar[rr]^{\pi}&&X\times_ZZ'\ar[d]\ar[rr]^g&&X\ar[d]_f\\
&&Z'\ar[rr]_{F^e}&&Z}$$
where $\pi:X'\rightarrow X\times_ZZ'$ is a resolution.  By the above argument, $f'$ has smooth generic fiber $F'$. By Theorem \ref{t-main},
$$\kappa(K_{X'}) \geq \kappa(K_{F'})+\kappa(K_{Z'}) = \kappa(K_{\tilde{F}})+\kappa(K_Z).$$

 Let $\sigma=g\pi: X' \rightarrow X$  be the natural composite map. By [\ref{cz13}, Theorem 2.4], there exists an effective $\sigma$-exceptional divisor $E$ on $X'$ such that
$$K_{X'/Z'} \leq \sigma^*K_{X/Z} + E.$$
Thus
$$K_{X'} + (p^{e} - 1)f'^*K_{Z'} \leq \sigma^*K_{X} + E.$$
We can assume $K_{Z'}$ is effective, thus
$$\kappa(K_X) = \kappa(\sigma^*K_{X} + E) \geq \kappa(K_{X'} + (p^{e} - 1)K_{Z'}) \geq \kappa(K_{X'}) \geq \kappa(K_{\tilde{F}})+\kappa(K_Z),$$
where the first ``$=$'' is from Covering Theorem \ref{ct}.

\end{proof}


\vspace{1cm}

DPMMS, Centre for Mathematical Sciences, University of Cambridge, Wilberforce Road, Cambridge CB3 0WB, UK\\
\email{cb496@dpmms.cam.ac.uk\\\\}

Hua Loo-Keng Key Laboratory of Mathematics, Academy of Mathematics and Systems Science,
Chinese Academy of Sciences, No. 55 Zhonguancun East Road, Haidian District, Beijing, 100190, P.R.China\\ 
yifeichen@amss.ac.cn\\\\  

College of Mathematics and Information Sciences, Shaanxi Normal University, Xi¡Çan 710062, P.R.China\\ 
lzhpkutju@gmail.com

\end{document}